\documentclass[submission,copyright,creativecommons]{eptcs}
\providecommand{\event}{QPL 2021} 
\def\publicationstatus{This is a revision of a paper with the same title,\\
published in EPTCS 343, 2021, pp.\ 119--131,\\
doi: \href{http://dx.doi.org/10.4204/EPTCS.343.6}{10.4204/EPTCS.343.6}}

\newtheorem{theorem}{Theorem}
\newtheorem{proposition}{Proposition}
\newtheorem{corollary}{Corollary}

\newtheorem{definition}[theorem]{Definition}
\newtheorem{proof}{Proof}

\newcommand{\+}[1]{\ensuremath{\mathbf{#1}}}
\usepackage{calrsfs}
\DeclareMathAlphabet{\pazocal}{OMS}{zplm}{m}{n}

\usepackage{amssymb}
\usepackage{tikz}
\usetikzlibrary{shapes.geometric}
\usepackage{pgfplots}
\pgfplotsset{compat=newest}
\usepackage{mathtools}
\DeclarePairedDelimiter{\expec}{\langle}{\rangle}

\usepackage{amsmath}
\usepackage{xfrac}
\usepackage{mathrsfs}

\title{Restricted Hidden Cardinality Constraints in Causal Models}
\author{Beata Zjawin
\institute{Perimeter Institute for Theoretical Physics 31 Caroline St. N, Waterloo, Ontario, Canada, N2L 2Y5}
\institute{ International Centre for Theory of Quantum Technologies,  University of Gdańsk, 80-308 Gdańsk, Poland}
\email{beata.zjawin@phdstud.ug.edu.pl}
\and
Elie Wolfe \qquad\qquad Robert W. Spekkens
\institute{Perimeter Institute for Theoretical Physics 31 Caroline St. N, Waterloo, Ontario, Canada, N2L 2Y5}}

\def\titlerunning{Restricted Hidden Cardinality Constraints}
\def\authorrunning{B. Zjawin, E. Wolfe \& R.W. Spekkens}
\begin{document}
\maketitle

\begin{abstract}
Causal models with unobserved variables impose nontrivial constraints on the distributions over the observed variables. When a common cause of two variables is unobserved, it is impossible to uncover the causal relation between them without making additional assumptions about the model. In this work, we consider causal models with a promise that unobserved variables have known cardinalities. We derive inequality constraints implied by $d$-separation in such models. Moreover, we explore the possibility of leveraging this result to study causal influence in models that involve quantum systems.
\end{abstract}

\section{Introduction}

In causal studies, systems of variables are described by \emph{causal models} \cite{pearl2000causality,spirtes2000causation}, which are composed of two elements: (i) the graphical representation of relationships between variables in a model, encoded in a directed acyclic graph, and (ii)~the~mathematical description of conditional probability distribution of each variable given its causal parents. When a causal model involves \emph{hidden} (i.e., unobserved) variables, any characterization of the model verifiable by observations should only include observed variables. Therefore, one of the objectives of causal inference is to eliminate all hidden variables from inequalities and equalities that describe the model. In principle, this can be achieved using the Tarski–Seidenberg quantifier elimination method \cite{geiger1999quantifier}. However, its complexity is such that only models with few variables can be solved using this technique, hence the reason for the many attempts to simplify the problem. For example, one can derive a set of incomplete constraints on probability distributions in latent causal models \cite{evans2012graphical,tian2012testable,kang2012polynomial,rosset2016nonlinear,tavakoli2016bell}. As these models impose highly nontrivial constraints on the probability distributions, Chaves et al.  \cite{chaves2014inferring} derive constraints expressed using \emph{entropies}. Interestingly, the entropic framework originates from the foundations of quantum mechanics \cite{chaves2013entropic,fritz2012beyond}. Furthermore, causal effects in latent causal structures can be estimated by limiting the causal influence between a confounder and its children \cite{rosset2017universal,steudel2015information}. Unfortunately, these methods only provide the necessary, but not sufficient, constraints for complete characterization of a~causal model.  

The field of causal inference is relevant to a wide range of disciplines. Curiously, quantum correlations seem to have no satisfactory explanation via classical causal models. The most famous example is Bell’s theorem \cite{bell1964einstein}. The underlying causal structure of a Bell experiment was investigated in \cite{wood2015lesson}. It was shown that it is impossible to explain Bell-inequality violations using a classical causal model without fine-tuning. As this result only holds for the framework of classical causal inference, it appears that quantum theory calls for a quantum generalization of this formalism. Many works attempt to describe causality in quantum systems by generalizing classical concepts to the quantum formalism, for example: quantum Reichenbach’s principle \cite{allen2017quantum}, quantum Bayes’ theorem \cite{leifer2013towards} and quantum randomized trial \cite{ried2015quantum}. We suggest that studying classical causal models with hidden variables can provide a valuable insight into quantum causality.

In this work, we analyze causal models with hidden variables that exhibit $d$-separation relations, under the assumption that the cardinality of the unobserved variables is known. We derive inequality constraints on the probability distribution over observed variables in such scenarios. We discuss possible applications of these constraints, both in classical causal inference and in quantum causality.

\subsection*{Outline}
Section~2 reviews the framework of graphical causal models and the definition of nonnegative rank. Section~3 contains the main results. It begins with description of nonnegative rank as a measure of dependence. Afterwards, it provides constraints implied by $d$-separation in causal models with hidden variables. Section~4 describes applications of constraints derived in Section~3. Section~5 shows that our constraints are sufficient, but not necessary for identifying direct causal influence. Finally, Section~6 discusses possible generalisations of our constraints to the quantum formalism. Lastly, Section~7 concludes the previous sections.

\section{Preliminaries}
\subsection*{Graphical Causal Models}

In this section, we briefly recall basic notions and important results from graph theory and causal inference.

Consider a system of \emph{causally} related events. We represent each event by a random variable and we denote the set of variables that corresponds to the system of events as $\+V$. In general, $\+V$ can contain both observed and unobserved variables. 

The variables in $\+V$, and relations between them, can be represented by a graph. Graphs consist of vertices and edges. We call a graph a \emph{causal graph}, if the edges that connect the vertices are directed and correspond to causal relations. For example, a causal graph $X \rightarrow Y$ indicates that the variable $X$ has a direct causal influence on $Y$. Here, $X$ is said to be a parent of $Y$. We denote all parents of a random variable $V$ in a graph $\mathcal G$ by $pa_{\mathcal G}(V)$.

In causal inference studies, relationships between variables are illustrated by \emph{Directed Acyclic Graphs} (DAGs) \footnote{A graph is directed, if all its edges are directed (they point in a certain direction which is marked with a single arrow). A directed graph is acyclic if there is no directed path from any variable back to itself.}. A \emph{causal model} consists of a DAG $\mathcal G$ over a set of variables $\+V$ and a joint probability distribution $P_{\+V}$. Causal models provide not only a probabilistic, but also a causal interpretation of the relationships between variables in a system. 

Properties of a DAG that represents variables in $\+V$ imply properties of the distribution $P_{\+V}$. The essential one is the Markov condition which states that $P_{\+V}$ can be decomposed as
\begin{align}
    P_{\+V}=\prod_{V \in {\+V}} P(V \mid pa_{\mathcal G}(V)).
\end{align}
Naturally, this condition can be tested only if all variables in $\+V$ are observed. It means that given a DAG $\mathcal G$ and a probability distribution $P_{\+V}$, with all variables in $\+V$ being observed, the Markov condition enables us to test the \emph{compatibility} of $\mathcal G$ and $P_{\+V}$. However, when some variables in $\+V$ are unobserved, testing compatibility is a hard task, as we already mentioned in the introduction.

A different property of DAGs that is important in causal inference studies is $d$-separation. Before we formally define it, let us recall the notion of blocking a path.

Let $p$ be a path in a DAG $\mathcal G$. We say that $p$ is blocked by a set of vertices $\+Z$ if and only if
\begin{itemize}
\item  $p$ contains a fork structure ($X
\leftarrow Z \rightarrow Y$) or a chain structure ($X \rightarrow Z \rightarrow Y$) in which the middle vertex $Z$ is in $\+Z$, or 
\item $p$ contains a collider structure ($X \rightarrow Z
\leftarrow Y$) in which neither the middle vertex $Z$ nor any of its descendants is in $\+Z$.
\end{itemize}

\begin{definition}[$d$-separation] \label{def:d-sep}
$\+X$ and $\+Y$ are said to be \emph{d-separated} by a set $\+Z$, $(\+X\perp_d~\+Y|\+Z)$, if and only if $\+Z$ blocks every path from a vertex in $\+X$ to a vertex in $\+Y$.
\end{definition}

Any $d$-separation relation $(\+X\perp_d~\+Y|\+Z)$ in $\mathcal G$ implies that $\+X$ and $\+Y$ are conditionally independent given $\+Z$, ($\+X \perp \+Y | \+Z$), in the corresponding distribution $P_{\+X\+Y\+Z}$.

Notice that testing whether conditional independence relations implied by $\mathcal G$ hold is only possible when all variables in $\+X$, $\+Y$ and $\+Z$ are observed. When dealing with unobserved variables, conditional independence relations are sometimes impossible to evaluate.

\subsection*{Nonnegative factorization}

We now review the definitions of nonnegative matrix factorization and nonnegative rank as they will play an important role in our results.

Nonnegative matrix factorization is a group of algorithms used for factorizing a nonnegative matrix (matrix with no negative entries) into two or more nonnegative matrices. The problem of performing a nonnegative matrix factorization is not always tractable, nevertheless numerical approximations of this task exist. This method is used in many fields of science, such as astronomy, bioinformatics and machine learning, as they study data with positive entries \cite{berne2007analysis,berry2007algorithms,murrell2011non}. Furthermore, the nonnegative rank, which we define below, plays a critical role in communication theory as it characterizes randomized communication complexity and randomized correlation complexity \cite{jain2013efficient,zhang2012quantum}. 

Formally, given a nonnegative matrix $\+M \in \mathbb{R}_+^{n\times m}$, a nonnegative factorization of $\+M$ means factorizing it into
\[
\+M =\+U \+V, 
\]
where $\+U \in \mathbb{R}_+^{n\times r}$ and $\+V \in \mathbb{R}_+^{r\times m}$ are nonnegative. The nonnegative rank of $\+M$, $rank_{+}[\+M]$, is the smallest number $r$ for which such a factorization can be found. Equivalently, the nonnegative rank of $\+M$ is the smallest number for which $\+M$ can be decomposed into a positive sum over $r$ nonnegative rank-1 matrices \cite{cohen1993nonnegative}. For the purpose of this paper, we formalize the definition of the nonnegative rank in a following way.

\begin{definition}[nonnegative rank]
  \label{def:nonnegative-rank}
  Let $\+P_{XY}$ be a ($|X| \times |Y|$)-dimensional nonnegative matrix. The nonnegative rank of $\+P_{XY}$ is defined as
  \begin{align}\label{eq:nnr}
  rank_+[\+P_{XY}] \equiv \min \{r \mid
  \exists (\+P_{X}^{(1)}, \+P_{Y}^{(1)}), \dots, (\+P_{X}^{(r)}, \+P_{Y}^{(r)})
  ~~\+P_{XY} = \sum_{z = 1}^r \+P_{X}^{(z)}\+P_{Y}^{(z)} \},
  \end{align}
where $\+P_{X}^{(z)}$ is a nonnegative $|X|$-dimensional row vector and each $\+P_{Y}^{(z)}$ is a nonnegative $|Y|$-dimensional column vector.
\end{definition}

The matrices in the set $\{\+P_{X}^{(z)}\+P_{Y}^{(z)}\}_z$ have rank 1, hence this definition coincides with the standard definition of the nonnegative rank.

\section{Rank constraints for causal models}

In this section, we explore a relation between the nonnegative rank of a probability distribution over two variables, and the causal relation between these variables.

\subsection*{Nonnegative rank as a measure of dependence}

We start by making a trivial observation that any joint probability distribution over two variables can be represented as a nonnegative matrix $\+P_{XY}$, with entries $\+P_{xy}=\+P_{X=x,Y=y}=P(X=x,Y=y)$. Notice that although we donote $\+P_{xy}$ using the bold font, it is a single element of the matrix $\+P_{XY}$. If $X$ and $Y$ are independent, $\+P_{XY}$ can be factorized as $\+P_{XY}=\+P_X \+P_Y$. Then, it follows from Definition~\ref{def:nonnegative-rank} that $\+P_{XY}$ has nonnegative rank $1$. Conversely, as the correlation between $X$ and $Y$ increases, so does the nonnegative rank of their distribution. It means that the nonnegative rank of a probability distribution can be interpreted as a measure of dependence between two variables. 

Consider an example in which the dependence between $X$ and $Y$ is only due to a third variable, denoted by $Z$. Let $Z$ take $r$ different values, i.e., the cardinality of $Z$ is $r$, which we denote $|Z|=r$. Throughout this paper we take $r$ to be the minimum cardinality of $Z$ for which $X$ and $Y$ are conditionally independent. If $X$ and $Y$ are conditionally independent given $Z$, the probability distribution must factorize as:
\begin{align}\label{eq:ci}
    \+P_{XY}=\sum_{z=1}^{r} \+P_{X|Z=z} \+P_{Y|Z=z} \+P_{Z=z}.
\end{align}
Here, $\+P_{X|Z=z}$ ($\+P_{Y|Z=z}$) is a conditional probability distribution over $X$ ($Y$) given that $Z$ takes a particular value $Z=z$. This factorization is equivalent to the one given in Eq.~\eqref{eq:nnr}, with $rank_{+}[\+P_{XY}]=r$. We summarize this result in the following Proposition.

\begin{proposition}\label{prop:nn}
Let $\+P_{XY}\in \mathbb{R}_+^{n\times m}$ represent a joint probability distribution on random variables $X$ and $Y$ that have finite cardinalities, $|X|=n$ and $|Y|=m$. The following are equivalent:
\begin{enumerate}
    \item There exists a nonnegative factorization of $\+P_{XY}$ with $rank_{+}[\+P_{XY}]=r$.
    \item $\+P_{XY}$ can be decomposed as in Eq.~\eqref{eq:ci} with $|Z|=r$.
\end{enumerate}
\end{proposition}
\begin{proof}
1 $\implies$ 2:
Suppose $\+P_{XY}$ factorizes as in Eq.~\eqref{eq:nnr}, with $\+P_{X}^{(z)}$ and $\+P_{Y}^{(z)}$ being unnormalized probability distributions. Define the sums
\[
Q^{(z)}_{X}=\sum_x \+P_{X=x}^{(z)}\,, \;\;\;\;\;\;\;\;\; Q^{(z)}_{Y}=\sum_y \+P_{Y=y}^{(z)}\,.
\]
Notice that the set of products $\{Q^{(z)}_{X}Q^{(z)}_{Y}\}_z \,$ form a valid probability distribution. We denote it as

\begin{align}\begin{split}
 \+ P_{Z=z} &=Q^{(z)}_{X}Q^{(z)}_{Y},\\
where \,\,\,\,\, \sum_z \+ P_{Z=z} &= \sum_{x,\;y,\;z} \+P_{X=x}^{(z)} \+P_{Y=y}^{(z)} = \sum_{x,\;y} \+P_{xy} =1.
\nonumber
\end{split}\end{align}
Further, define the following vectors:
\begin{align}\begin{split}
 \+C_{X}^{(z)} =\frac{\+P_{X}^{(z)}}{Q^{(z)}_{X}} \,, \;\;\;\;\;\;\;\;\;
 \+C_{Y}^{(z)} =\frac{\+P_{Y}^{(z)}}{Q^{(z)}_{Y}}.
\nonumber
\end{split}\end{align}
We can now decompose $\+P_{XY}$ as
\[
\+P_{XY}=\sum_z \+P_{X}^{(z)} \+P_{Y}^{(z)}  = \sum_z \+ \+C_{X}^{(z)}\, \+C_{Y}^{(z)}\, (\,Q^{(z)}_{X} Q^{(z)}_{Y}\,).
\]
It is now clear that one can define the following conditional probabilities
\begin{align}\begin{split}
 \+P_{X|Z=z}  \equiv \+C_{X}^{(z)}\,, \;\;\;\;\;\;\;\;\;
 \+P_{Y|Z=z} \equiv \+C_{Y}^{(z)},
\nonumber
\end{split}\end{align}
and decompose $\+P_{XY}$ as
\[
\+P_{XY}= \sum_z \+ P_{X|Z=z}\+ P_{Y|Z=z}\+ P_{Z=z},
\]
which is the factorization from Eq.~\eqref{eq:ci}.
\newline
\noindent
2 $\implies$ 1: Assume $\+P_{XY}$ can be factorized as in Eq.~\eqref{eq:ci}. Then, since $\+P_{Z=z}$ is just a positive number, it follows trivially that $rank_{+}[\+P_{XY}]=r$.
\end{proof}
This proposition was also stated in Refs. \cite{gonzalo2017cond}, \cite[Appendix 7.5]{kocaoglu2018applications} and \cite[Proposition 2]{ding2008equivalence}.

Recall that we assumed that $r$ is the minimum cardinality of $Z$ for which $X$ and $Y$ are conditionally independent. The direction (2 $\implies$ 1) of Proposition~\ref{prop:nn} is true due to this assumption. In general, the cardinality of $Z$ can be greater than $rank_{+}[\+P_{XY}]$, which we exploit in the next Subsection.

\subsection*{Restricted cardinality constraints}

Consider the example from the previous Subsection, where $Z$ is assumed to be the only source of dependence between $X$ and $Y$. If $Z$ is an observed variable, it is straightforward to verify whether this hypothesis holds. However, if $Z$ is unobserved, constraints implied by conditional independence relations are impossible to directly evaluate. We now show that knowledge about the cardinality of $Z$ is enough to test conditional independence relations in some cases.

Let us interpret Proposition~\ref{prop:nn} from the perspective of causal models. If $\+P_{XY}$ can be factorized as in Eq.~\eqref{eq:ci}, then there exist a hypothetical variable $Z$, with cardinality $|Z|=r$, and probability distribution $\+P_Z$ that makes $X$ and $Y$ conditionally independent. In graphical causal models, the notion of conditional independence is captured by $d$-separation. If two variables $X$ and $Y$ are conditionally independent given $Z$, a DAG that illustrates the causal structure of $X$, $Y$ and $Z$ must satisfy the following two properties. First, there must be no direct causal influence between $X$ and $Y$ (no direct arrow form $X$ to $Y$, or from $Y$ to $X$). Second, $Z$ must $d$-separate $X$ and $Y$. Three examples of causal structures that satisfy these properties are depicted in Figure~\ref{fig:d-sep}. Circles correspond to observed variables, and dashed circles correspond to unobserved variables. In Figure~\ref{fig:d-sep}~(a), $Z$ is a common cause of $X$ and $Y$. In Figure~\ref{fig:d-sep}~(b), $Z$ is a causal mediary between $X$ and $Y$. Finally, in Figure~\ref{fig:d-sep}~(c), a $d$-separating set $\+Z=\{Z_1,Z_2\}$ consists of both a common cause and a causal mediary. In the first two examples, the cardinalities of the $d$-separating variables are given by $|Z|$. In the third example, the cardinality of the $d$-separating set is equal to $|Z_1|~\times~|Z_2|$.

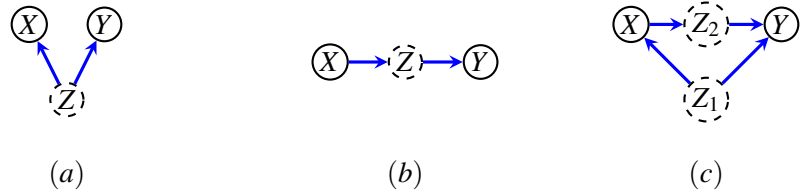
\begin{figure}[h]
  \centering
  \begin{tikzpicture}[>=stealth, node distance=1.0cm]
%    \tikzstyle{vertex} = [
%    draw, very thick, minimum size=7.5mm, inner sep=0pt
%    ]
\tikzstyle{vertex} = [draw, thick, ellipse, minimum size=4.0mm,
            inner sep=1pt]
             
    \tikzstyle{edge} = [
    ->, blue, very thick
    ]

	\begin{scope}
    \node[vertex, circle] (a) {$X$};
    \node[vertex, circle] (b) [right of=a] {$Y$};
    \node[vertex, circle,dashed] (c) [below of=a, xshift=0.5cm] {$Z$};

    \draw[edge] (c) to (a);
    \draw[edge] (c) to (b);

    \node[below of=c] (l) {$(a)$};
  \end{scope}

	\begin{scope}[xshift=4.0cm,yshift=-0.5cm] 
    \node[vertex, circle] (a) {$X$};
    \node[vertex, circle,dashed] (m) [right of=a] {$Z$};
    \node[vertex, circle] (b) [right of=m] {$Y$};

    \draw[edge] (a) to (m);
    \draw[edge] (m) to (b);

    \node[below of=m, yshift=-0.5cm] (l) {$(b)$};
  \end{scope}

	\begin{scope}[xshift=8.0cm]
	    \node[vertex, circle] (a) {$X$};
	    \node[vertex, circle,dashed] (m) [right of=a] {$Z_2$};
    \node[vertex, circle] (b) [right of=m] {$Y$};
    \node[vertex, circle,dashed] (c) [below of=m] {$Z_1$};

    \draw[edge] (c) to (a);
    \draw[edge] (c) to (b);
    \draw[edge] (a) to (m);
    \draw[edge] (m) to (b);

    \node[below of=c] (l) {$(c)$};
  \end{scope}

\end{tikzpicture}

  \caption{
    Examples of $d$-separating variables.
  }
  \label{fig:d-sep}
\end{figure}

From Proposition~\ref{prop:nn}, the variable that $d$-separates $X$ and $Y$ must have at least cardinality $r$, which is given by $rank_{+}[\+P_{XY}]$. This result is summarized in the corollary below.

\begin{corollary} \label{cor}
The nonnegative rank of the matrix $\+P_{XY}$ whose components are $\+P_{xy}=P(X=x,Y=y)$ constitutes a lower bound for the cardinality of a variable $\+Z$ that d-separates $\+X$ and $\+Y$,
\begin{align}\label{eq:cor}
    rank_{+}[\+P_{XY}] \leq |\+Z|.
\end{align}
\end{corollary}
The constraint given in Eq.~\eqref{eq:cor} enables to verify whether conditional independence relations implied by a DAG hold, even when the $d$-separating variable is unobserved. The knowledge about $Z$ required to evaluate the bound is only its cardinality. In the next section, we give examples that use Corollary~\ref{cor} to test conditional independence relations in DAGs with latent variables. 

\section{Applications}
In this section, we give two applications of Corollary~\ref{cor}. First, we describe its potential to improve causal discovery algorithms. In particular, we consider a scenario in which our constraint witnesses direct causal influence between two variables in a presence of a confounder. Second, we reprove two results from communication theory: characterization of randomized correlation complexity and randomized communication complexity. 

\subsection*{Witnessing Causal Influence}

Imagine that two ternary events, $X$ and $Y$ ($|X|=|Y|=3$), are observed to exhibit some statistical dependence. Additionally, assume they share an unobserved common cause $U$ which is known to have cardinality $|U|=2$. Two possible hypotheses regarding the causal relations between $X$, $Y$ and $U$ are illustrated in Fig.~\ref{fig:identifying}.
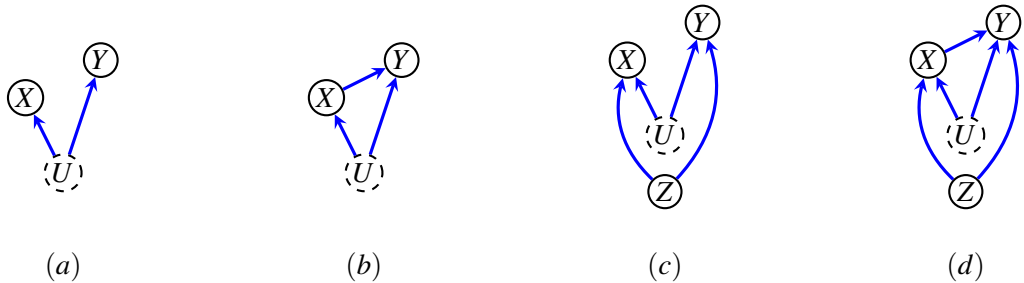
\begin{figure}[h]
  \centering
  \begin{tikzpicture}[>=stealth, node distance=1.0cm]
%    \tikzstyle{vertex} = [
%    draw, very thick, minimum size=7.5mm, inner sep=0pt
%    ]
\tikzstyle{vertex} = [draw, thick, ellipse, minimum size=4.0mm,
            inner sep=1pt]
             
    \tikzstyle{edge} = [
    ->, blue, very thick
    ]
    
    \begin{scope}[yshift=-0.5cm]
    \node[vertex, circle] (x) {$X$};
    \node[vertex, circle] (y) [right of=x,yshift=0.5cm] {$Y$};
    \node[vertex, dashed, circle] (h) [below of=x, xshift=0.5cm] {$U$};

    \draw[edge] (h) to (x);
    \draw[edge] (h) to (y);    
    
    \node[below of=h,yshift=-0.25cm] (l) {$(a)$};
    \end{scope}
    
    \begin{scope}[xshift=4.0cm,yshift=-0.5cm]
        \node[vertex, circle] (x) {$X$};
    \node[vertex, circle] (y) [right of=x,yshift=0.5cm] {$Y$};
    \node[vertex, dashed, circle] (h) [below of=x, xshift=0.5cm] {$U$};

    \draw[edge] (x) to (y);
    \draw[edge] (h) to (x);
    \draw[edge] (h) to (y); 
    
    \node[below of=h,yshift=-0.25cm] (l) {$(b)$};
    \end{scope}
    
    \begin{scope}[xshift=8.0cm]
        \node[vertex, circle] (x) {$X$};
    \node[vertex, circle] (y) [right of=x,yshift=0.5cm] {$Y$};
    \node[vertex, dashed, circle] (h) [below of=x, xshift=0.5cm] {$U$};
    \node[vertex, circle] (z) [below of=x, xshift=0.5cm,yshift=-0.75cm] {$Z$};

    \draw[edge] (h) to (x);
    \draw[edge] (h) to (y); 
    \draw[edge, bend left=30] (z) to (x); 
    \draw[edge, bend right=30] (z) to (y); 
    
    \node[below of=z] (l) {$(c)$};
    \end{scope}
    
    \begin{scope}[xshift=12.0cm]
        \node[vertex, circle] (x) {$X$};
    \node[vertex, circle] (y) [right of=x,yshift=0.5cm] {$Y$};
    \node[vertex, dashed, circle] (h) [below of=x, xshift=0.5cm] {$U$};
    \node[vertex, circle] (z) [below of=x, xshift=0.5cm,,yshift=-0.75cm] {$Z$};

    \draw[edge] (x) to (y);
    \draw[edge] (h) to (x);
    \draw[edge] (h) to (y); 
        \draw[edge, bend left=30] (z) to (x); 
    \draw[edge, bend right=30] (z) to (y); 
    
    \node[below of=z] (l) {$(d)$};
    \end{scope}
\end{tikzpicture}

  \caption{
    Witnessing direct causal influence in a presence of a confounder.
  }
  \label{fig:identifying}
\end{figure}
In Fig.~\ref{fig:identifying}~(a), one believes the statistical dependence is purely due to a common cause, whereas in Fig.~\ref{fig:identifying}~(b), $X$ additionally has a direct causal influence on $Y$. 

In Fig.~\ref{fig:identifying}~(a), $U$ $d$-separates $X$ and $Y$. From Corollary~\ref{cor}, we know that
\[
rank_{+}[\+P_{XY}] \leq |U|.
\]
Therefore, if $rank_{+}[\+P_{XY}] > 2$, one can discard the hypothesis of Fig.~\ref{fig:identifying}~(a) and conclude that there must be a direct causal influence of $X$ on $Y$ as in Fig.~\ref{fig:identifying}~(b). 

As an example, suppose $\+P_{XY}$ is given by the probability distribution from Table~\ref{tab:dist}. In this case, the nonnegative rank of $\+P_{XY}$ is equal to 3. One can certify that $rank_{+}[\+P_{XY}]$ is greater than $|U|$, and only the causal hypothesis of Fig.~\ref{fig:identifying}~(b) is compatible with this probability distribution. In this example, observing $rank_{+}[\+P_{XY}]=3$ \emph{witnesses} direct influence between $X$ and $Y$.

\begin{table}[h]
\begin{center}
\begin{tabular}{ c c | c c c  }
 & & &  Y &  \\
 & & 0 & 1 & 2  \\\hline
 & 0 & 0 &  $\sfrac{1}{6}$  & $\sfrac{1}{6}$\\
 X & 1 & $\sfrac{1}{6}$ & 0 & $\sfrac{1}{6}$ \\
 & 2 & $\sfrac{1}{6}$ & $\sfrac{1}{6}$ & 0 \\

\end{tabular}
\end{center}
\caption{An example of a probability distribution $\+P_{XY}$, with $|X|=|Y|=3$.
}
\label{tab:dist}
\end{table}

The example above shows the essence of Corollary~\ref{cor}. It captures the fact that strong correlation between high cardinality variables, when no direct influence between them exists, must be explained by a high cardinality common cause. This dependence was also noticed in Ref.~\cite[Corollary 1]{kocaoglu2018applications}, where the authors show that the graphs illustrated in Fig.~\ref{fig:identifying}~(a) and Fig.~\ref{fig:identifying}~(b) may sometimes be distinguished under the presumption that the confounder has small cardinality, i.e., $|U| < \,\textrm{min}\,\{|X|,|Y|\}.$

Corollary~\ref{cor} can also be applied to scenarios where the common cause is \emph{partially} observed. We define a partially observed set as a set in which only some variables are observed. An example of a partially observed common cause $\+C=\{Z,U\}$ is illustrated in Fig.~\ref{fig:identifying}. In Fig.~\ref{fig:identifying}~(c), $X$ and $Y$ share two common causes: an unobserved one, denoted by $U$, and an observed one, denoted by $Z$. In Fig.~\ref{fig:identifying}~(d), $X$ additionally influences $Y$. 

In Fig.~\ref{fig:identifying}~(c), $Z$ is an observed variable, and we can condition the values of $X$ and $Y$ on its particular value, $Z=z$. Let $\+P^{(z)}_{XY}$ denote the probability distribution over $X$ and $Y$ conditioned on a particular value $z$ of $Z$, $\+P^{z}_{xy}~=~P(X=x,Y=~y|Z=z)$. After conditioning, $U$ $d$-separates $X$ and $Y$ (given $Z=z$) in Fig.~\ref{fig:identifying}~(c). The constraint given by Corollary~\ref{cor} in this case is
\[ \forall z: \,\,\,rank_{+}[\+P^{(z)}_{XY}] \leq |U|.
\]
Given $\+P^{(z)}_{XY}$ for all $z$, any violation of this condition witnesses direct causal influence between $X$ and $Y$, just like in the example with an unobserved common cause.

Corollary~\ref{cor} can also be used to solve a converse problem. Suppose a distribution $\+P_{XY}$ is observed, where variables $X$ and $Y$ are known to share an unobserved common cause $Z$ (or more generally, are $d$-separated by $Z$). If nothing is known about $Z$, constraint from Eq.~\eqref{eq:cor} can be used to lower bound its cardinality.

\subsection*{Distribution generation}

Recall a distribution generation protocol that is often discussed in communication theory \cite{zhang2012quantum,jain2013efficient}. Suppose two parties, Alice and Bob, are responsible for generating random variables $X$ and $Y$ respectively, such that $X$ and $Y$ have a prespecified joint distribution $\+P_{XY}$. Our goal is to quantify how much information must be communicated or shared to generate $\+P_{XY}$.

In the literature about this problem\, two scenarios have been developed and studied separately. In the first scenario, Alice and Bob each receive
information from a third party before they generate $X$ and $Y$. We denote the information they share by $Z$. In this protocol, no communication between Alice and Bob is allowed. The information $Z$ distributed this way is often referred to as a random seed. The minimum size of a seed that must be given
to Alice and Bob in order to produce the required distribution $\+P_{XY}$ is
called the \textit{randomized correlation complexity} of the joint distribution, and is denoted $RCorr[\+P_{XY}]$. Here, a size of a seed is the number of bits that is distributed to each party, and is given by $\text{log}_2 (|Z|)$. In the second scenario, one-way communication between Alice and Bob is allowed. After generating $X$, Alice sends a message to Bob through an intermediary, which we denote $M$. The minimum number of bits Alice needs to communicate to Bob such that they can produce the target distribution is called the \textit{randomized communication complexity}, and is denoted $RComm[\+P_{XY}]$. The number of bits
communicated between the parties is given by $\text{log}_2 (|M|)$.

It has been shown that $RCorr[\+P_{XY}]$ and $RComm[\+P_{XY}]$ are equal, and are fully characterized by $\text{log}_2(rank_+[ \+P_{XY}])$ \cite{zhang2012quantum}. We now show that given Corollary~\ref{cor}, this result follows trivially from a causal analysis of the problem. To begin, we note that in the first scenario, the seed $Z$ is distributed to both Alice and Bob, hence it can affect both $X$ and $Y$. Graphically, this scenario corresponds to Fig.~\ref{fig:d-sep}~(a), where $Z$ is a common cause of $X$ and $Y$. In the second scenario, the message $M$ can depend on the random variable $X$, which in
turn affects how Bob generates $Y$. This scenario is illustrated by a causal structure in Fig.~\ref{fig:d-sep}~(b), with $M$ playing the role of $Z$ in the figure.

We now observe that $X$ and $Y$ are d-separated by $Z$ and $M$ respectively in the two scenarios. Corollary~\ref{cor} implies that
\[
|Z| \ge rank_+[\+P_{XY}] \qquad \text{and} \qquad |M| \ge
rank_+[\+P_{XY}].
\]
It follows that
\[
\text{log}_2(|Z|) \ge \text{log}_2(rank_+[\+P_{XY}]) \qquad \text{and} \qquad \text{log}_2(|M|) \ge \text{log}_2(rank_+[\+P_{XY}]).
\]
Note that $RCorr[\+P_{XY}]$ and $RComm[\+P_{XY}]$ correspond to the \emph{minimum} size of a seed and a message, respectively, needed to generate $\+P_{XY}$. Therefore, they are given by the lower bound of the above inequalities. Moreover,
Corollary~\ref{cor} does not distinguish a common cause $Z$
from a causal mediary $M$, as they both $d$-separate $X$ and $Y$ in this example. This result yields both the equivalence of $RCorr[\+P_{XY}]$ and $RComm[\+P_{XY}]$, and the fact that they are fully characterized by $\text{log}_2(rank_+[\+P_{XY}])$. This example shows that adopting a causal perspective can simplify results in communication theory.

Finally, we analyze a more involved protocol, in which Alice and Bob share a seed $Z_1$, \textit{and} Alice can communicate to Bob through a mediary $Z_2$. This scenario is illustrated in Fig.~\ref{fig:d-sep}~(c). The cardinality of the
d-separating set $\+Z=\{Z_1, Z_2\}$ is equal to $|Z_1| \times |Z_2|$. It follows from Corollary~\ref{cor} that 
\[
|Z_1| \times |Z_2| \ge rank_+[\+P_{XY}] \qquad \text{and} \qquad \text{log}_2(|Z_1| \times |Z_2|) \ge \text{log}_2(rank_+[\+P_{XY}]).
\]
We can intuitively view the constraint on $|Z_1| \times |Z_2|$ as defining a trade-off between the amount of information shared through a seed and the amount transferred using communication. This perspective enables studying relatively complex problems in which the size of the seed or the capacity of the communication channel is constrained. Similarly, if
one of these cardinalities is unobserved, information about the other can be used to bound it, using Corollary~\ref{cor}.

\section{A Specific Case with Additional Constraints}

The inequality given in Corollary~\ref{cor} gives a sufficient, but not a necessary condition for identifying causal relations between variables. In this Section, we give an example of a correlation that satisfies the constraints implied by the $d$-separation relations in a model, but nevertheless imposes new constraints that require additional causal mechanisms. 

Let us revisit the causal hypotheses illustrated in Fig.~\ref{fig:identifying}~(c) and Fig.~\ref{fig:identifying}~(d). Assume $X$, $Y$ and $Z$ are binary observed variables, and $U$ is a binary unobserved variable. In this scenario, $\+P^{(z)}_{XY}$ is a $2 \times 2$ matrix and the constraint given in Corollary~\ref{cor} is \emph{always} satisfied. This is the case because $rank_{+}[\+P^{(z)}_{XY}] \leq 2$ and $|U|=2$. However, this is not enough to conclude that the causal relations between $X$, $Y$, $Z$ and $U$ are those illustrated by Fig.~\ref{fig:identifying}~(c). In this Section, we give an example of a probability distribution $\+P^{(z)}_{XY}$ that satisfies the above conditions, and requires a direct causal influence as in Fig.~\ref{fig:identifying}~(d). First, we consider a specific probability distribution $\+P^{(z)}_{XY}$, with $|X|=|Y|=|Z|=|U|=2$, under the causal hypothesis of Fig.~\ref{fig:identifying}~(c). Then, we show that the causal structure of Fig.~\ref{fig:identifying}~(c) implies some nontrivial constraints that must be satisfied in order to explain the specific form of $\+P^{(z)}_{XY}$.

We start with a reminder that an intuitive way of describing an outcome of a random variable is with its expectation value. Let $\+V=\{V_1,...,V_n\}$ be a set of $n$ binary variables that each take a value in $\{-1,1\}$. The conversion formula between the probability distribution $P(\+V)=P(V_1,...V_n)$ and the expectation values of variables in $\+V$ can be written as 
\begin{align}\label{eq:formula_exp}
P(V_1,...V_n)=\frac{1}{2^n} \left(1+ \sum_{k=1}^n \; \sum_{1 \leq i_1 < ... < i_k \leq n} v_{i_1}...v_{i_k} \expec{V_{i_1} ... V_{i_k}}\right).
\end{align}

Let us consider a special case when $X$ and $Y$ are observed to be perfectly positively correlated for both possible values of $Z$, 
\begin{align}\label{eq:XYcorr}
    \expec{XY|Z=z}=1.
\end{align}
Assume that this perfect correlation arises due to the causal mechanism encoded in Fig.~\ref{fig:identifying}~(c). Eq.~\eqref{eq:XYcorr} implies that
\begin{align}\label{eq:XY1}
\begin{split}
& \text{(a)} \;\;\;\; \expec{X|Z,U}= \pm 1; \text{and}
\\& \text{(b)} \;\;\;\; \expec{Y|Z,U}=~\expec{X|Z,U},
\end{split}\end{align}
for both values of $Z$ and $U$, that is, $X$ and $Y$ need to depend deterministically on $Z$ and $U$ (Eq.~(\ref{eq:XY1}a)) in the same way (Eq.~(\ref{eq:XY1}b)).

Our goal is to derive constraints on $\expec{X|Z=z}$ that arise due to this condition. If any nontrivial constraints exist, the causal hypothesis of Fig.~\ref{fig:identifying}~(c) explains the perfect correlation between $X$ and $Y$ \emph{only if} they are satisfied. Violation of these constraints calls for an additional causal mechanism (additional to ones in Fig.~\ref{fig:identifying}~(c)) that explains the dependence between $X$ and $Y$. 

In order for the constraints on $\expec{X|Z=z}$ to be verifiable, they can only depend on the observed variables. We perform a quantifier elimination using the known constraints to eliminate the unobserved variable from the set of inequalities and equalities. We find that either
\begin{align}\begin{split}
&{\expec{X|z=1}= \pm \expec{X|z=-1}}
\\& \text{or Min}[\expec{X|z=1},\expec{X|z=-1}]=-1
\\& \text{or Max}[\expec{X|z=1},\expec{X|z=-1}]=1.
\label{eq:expecXconstraints}
\nonumber
\end{split}\end{align}
It then follows from $\expec{XY|Z=z}=1$ that 
\begin{align}\begin{split}
&{\expec{Y|z=1}= \pm \expec{Y|z=-1}}
\\& \text{or Min}[\expec{Y|z=1},\expec{Y|z=-1}]=-1
\\& \text{or Max}[\expec{Y|z=1},\expec{Y|z=-1}]=1.
\nonumber
\end{split}\end{align}
These constraints show that for the special case being considered, one of the following conditions must be true for the marginal expectation values of $X$ conditional on $Z$: 
\begin{enumerate}
\item $\expec{X|Z=z}$  has the same magnitude for $z=1$ and $z=-1$ (it is equal up to a sign); or  
\item $\expec{X|Z=z}$ is deterministic and does not depend on $U$ at all for one value of $Z=z$. 
\end{enumerate}
The same holds for $\expec{Y|Z=z}$. If any of these conditions on $\expec{X|Z=z}$ and $\expec{Y|Z=z}$ are violated, then we conclude that perfect positive correlation between $X$ and $Y$ cannot be explained by the causal structure encoded in Fig.~\ref{fig:identifying}~(c). An additional causal influence between $X$ and $Y$ is needed to explain the considered correlation, for example the one in Fig.~\ref{fig:identifying}~(d).

This is an example of a correlation between $X$ and $Y$ wherein the constraint induced by Corollary~\ref{cor} is satisfied for the $d$-separation relation pertinent to the causal hypothesis of Fig.~\ref{fig:identifying}~(c), but nonetheless this correlation is incompatible with this causal hypothesis. 

\section{Towards Quantum Causality}

We now give a different perspective on the motivation behind studying causal models with hidden common causes. In this Section, we show that DAGs with unobserved common causes can act as classical analogues of models in which a \emph{quantum} common cause is shared. 

Consider a scenario in which two parties share a source of entangled quantum systems and use it to output a classical correlation $\+P_{AB}$. The structure of this protocol is illustrated in Fig.~\ref{fig:Q}. The quantum seed that is distributed to Alice and Bob is denoted by $\rho_C$, quantum systems are represented by double wires and classical variables are represented by single wires. Under the assumption that only one copy of $\rho_C$ is available in this protocol, some form of a quantum copy operation must be performed on it before the distribution. As quantum systems cannot be cloned, the quantum decomposition operation introduced in \cite{allen2017quantum} is applied. It is represented with the white dot symbol and implies the decomposition of the Hilbert space $\mathscr{H}_C$ into a direct sum of tensor products. This operation decomposes $\rho_C$  into factors on $\mathscr{H}_{C_{B}}$ carried to Bob and $\mathscr{H}_{C_A}$ carried to Alice. The outputs of this operation are denoted by $\rho_A$ and $\rho_B$. Assume $\+P_{AB}$ is generated by one of the two protocols. The first protocol is illustrated in Fig.~\ref{fig:Q}~(a). Here, Bob and Alice perform measurements $M_A$ and $M_B$ on their share of the system $\rho_C$ to generate classical variables $A$ and $B$, respectively. The second protocol is encoded in Fig.~\ref{fig:Q}~(b). It allows for classical communication from Bob to Alice before Alice outputs $A$. In this protocol, the value of $A$ is an output of a gate $U_{AB}$ that depends on the value of $B$.
\begin{figure}[h]
  \centering
  \includegraphics[width=0.4\textwidth]{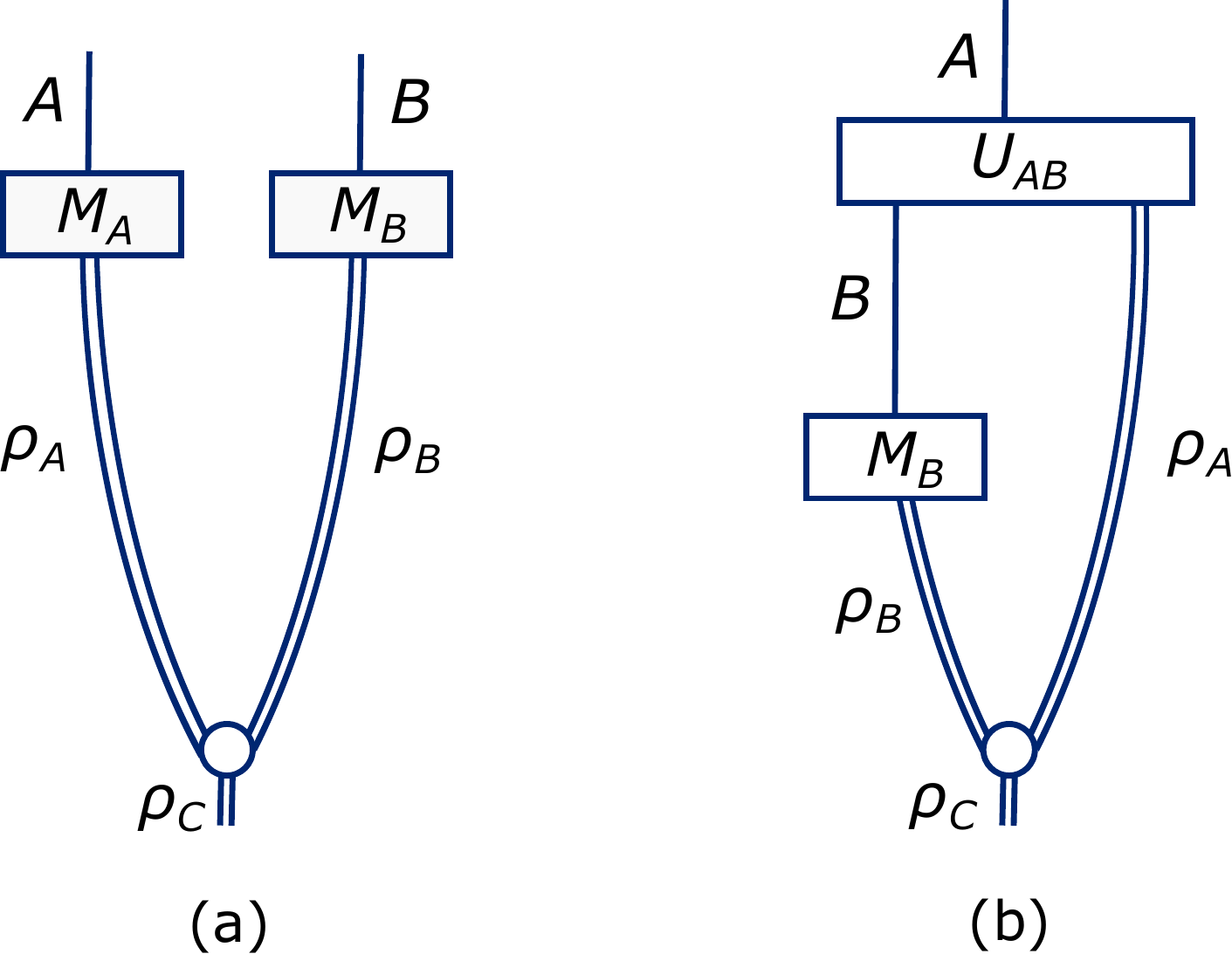}
  \caption{
    Circuit representation of correlation generation protocols.
  }
  \label{fig:Q}
\end{figure}

The main challenge of analyzing these protocols is the restriction on acquiring information about quantum systems relative to classical systems. Once a quantum state is sharply measured, it is projected to a post-measurement state that corresponds to the choice of measurement. This feature of quantum theory prevents the observer from making passive observations of quantum common causes. However, the observer can possess some limited information about the quantum system, for example, its dimension. This situation is analogous to a classical causal model involving an unobserved common cause with restricted cardinality. In such models, the observer has limited knowledge about the classical common cause, namely, its cardinality. This perspective on the problem allows to see the similarities between classical causal models illustrated in Figs.~\ref{fig:identifying}~(a) and (b), and circuits in Figs.~\ref{fig:Q}~(a) and (b). 

Although Corollary~\ref{cor} only concerns classical causal models, a similar concept is used in communication theory \cite{zhang2012quantum,fiorini2012linear}, where models with a quantum seed can be characterized using positive semidefinite rank \cite{jain2013efficient}. Therefore, Corollary~\ref{cor} has the potential to be further extended to models that include quantum systems. Specifically, knowledge about the dimension of a quantum common cause could play a role similar to that of the cardinality of a classical common cause. Formally, the minimum dimension of the local Hilbert space of a shared entangled state needed to generate $\+P_{AB}$ is defined as quantum correlation complexity. It is characterized by the positive semidefinite rank of $\+P_{AB}$ \cite{jain2013efficient}, where positive semidefinite rank is defined as follows. Let $\text{dim}(\rho_A)$ and $\text{dim}(\rho_B)$ be the dimensions of the local Hilbert space of Alice and Bob, respectively.

\begin{definition}[positive semidefinite rank]
  \label{def:psd-rank}
  Let $\+P_{AB}$ be a $\text{dim}(\rho_A) \times \text{dim}(\rho_B)$-dimensional nonnegative matrix. The positive semidefinite rank of $\+P_{AB}$  is defined as
  \begin{align}
  \begin{split}
  rank_{PSD}[\+P_{AB}] \equiv & \min \{r \mid
  \exists (\+E_1,\+F_{1}),..., (\+E_{\text{dim}(\rho_A)}, \+F_{\text{dim}(\rho_B)}) \\
  & P_{ab} = \mathrm{tr}\{E_a F_b \}~~\forall a \in \{1...\text{dim}(\rho_A)\}, b \in \{1...\text{dim}(\rho_B)\},
  \end{split}
  \end{align}
where $\+E_a$ and $\+F_b$ are $r \times r$ positive semidefinite matrices.
\end{definition}

Assume $\text{dim}(\rho_A)=\text{dim}(\rho_B)=d$. It follows that for the hypothesis illustrated in Fig.~\ref{fig:Q}~(a) to hold, the relation
\[
d \geq rank_{PSD}[\+P_{AB}]
\]
must be satisfied. This relation enables us to analyse Figs.~\ref{fig:Q}~(a) and (b) in a similar manner to the analysis of Figs.~\ref{fig:identifying}~(a) and (b) we conducted in the subsection \emph{Witnessing Causal Influence}.

One interesting property of the positive semidefinite rank is that for any matrix $\+M$
\[
    rank_{PSD}[\+M] \leq rank_{+}[\+M].
\]
For example, for the probability distribution $\+P_{XY}$ defined in Table~\ref{tab:dist}, $rank_{+}[\+P_{XY}]=3$, while $rank_{PSD}[\+P_{XY}]=2$ (see Example 2.6 in \cite{fawzi2015positive}). It means that the dimension of a quantum resource needed to generate $\+P_{XY}$ is smaller than the cardinality of a classical resource needed for this task. Examples of this form have been studied before \cite{zhang2012quantum}, and they show an advantage of quantum resources over classical ones.

We finish by noting that a similar approach can be used to analyse a protocol in which a quantum message is sent from Bob to Alice (the quantum analogue of Fig.~\ref{fig:d-sep}~(b)). Similarities between classical and quantum protocols for correlation generation suggest that classical results, like the one given in Corollary~\ref{cor}, can be used to set the stage for the quantum equivalent of causal analysis.

\section{Conclusions}

We derived universal constraints for causal models in which the cardinality of an unobserved $d$-separating variable is known. While the majority of existing methods of constraining distributions compatible with latent causal structures do not use any information about unobserved variables, our result relies on having knowledge of the cardinality of an unobserved variable. We showed that our result has applications in causal analysis. Furthermore, we recognized the analogy between classical causal models with hidden common causes and models that involve quantum systems. As we identify a new way of bounding causal effects in models with hidden variables, we set the stage for quantum extensions of this formulation. 

\section*{Acknowledgments}
B.Z. acknowledges useful discussions with Noam Finkelstein. This research was supported by Perimeter Institute for Theoretical Physics. Research at Perimeter Institute is supported in part by the Government of Canada through the Department of Innovation, Science and Economic Development Canada and by the Province of Ontario through the Ministry of Colleges and Universities.
\nocite{*}
\bibliographystyle{eptcs}
\bibliography{generic}
\end{document}